\documentclass[10pt]{amsart}
\usepackage{latexsym}
\usepackage[dvips]{color}
\usepackage{float}
\usepackage{afterpage} 
\usepackage [pdftex]{graphicx}
\usepackage{epsfig}
\usepackage{amsmath}
\usepackage{amsthm}
\usepackage{amsfonts}
\usepackage{amssymb}
\usepackage{setspace}
\usepackage{subfigure}
\usepackage{tikz}
\usetikzlibrary{arrows}
\usetikzlibrary{arrows.meta}
\usetikzlibrary{matrix}
\usepackage{subfigure}

\newtheorem{definition}{Definition}
\newtheorem{proposition}{Proposition}

\newtheorem{theorem}{Theorem}

\begin{document}

\title[Groupoids and wreath products of musical transformations]{Groupoids and wreath products of musical transformations: a categorical approach from poly-Klumpenhouwer networks.}
\author{Alexandre Popoff}
\address{119 Rue de Montreuil, 75011 Paris}
\email{al.popoff@free.fr}

\author{Moreno Andreatta}
\address{IRCAM/CNRS/UPMC and IRMA-GREAM, Universit\'e de Strasbourg}
\email{Moreno.Andreatta@ircam.fr, andreatta@math.unistra.fr}

\author{Andr\'ee Ehresmann}
\address{Universit\'e de Picardie, LAMFA}
\email{andree.ehresmann@u-picardie.fr}


\begin{abstract}
Transformational music theory, pioneered by the work of Lewin, shifts the music-theoretical and analytical focus from the ``object-oriented'' musical content to an operational musical process, in which transformations between musical elements are emphasized. In the original framework of Lewin, the set of transformations often form a group, with a corresponding group action on a given set of musical objects. Klumpenhouwer networks have been introduced based on this framework: they are informally labelled graphs, the labels of the vertices being pitch classes, and the labels of the arrows being transformations that maps the corresponding pitch classes. Klumpenhouwer networks have been recently formalized and generalized in a categorical setting, called poly-Klumpenhouwer networks. This work proposes a new groupoid-based approach to transformational music theory, in which transformations of PK-nets are considered rather than ordinary sets of musical objects. We show how groupoids of musical transformations can be constructed, and an application of their use in post-tonal music analysis with Berg's \textit{Four pieces for clarinet and piano, Op. 5/2}. In a second part, we show how groupoids are linked to wreath products (which feature prominently in transformational music analysis) through the notion groupoid bisections.
\end{abstract}

\maketitle

\section{Groupoids of musical transformations}

The recent field of transformational music theory, pioneered by the work of Lewin \cite{Lewin_1982, Lewin_1987}, shifts the music-theoretical and analytical focus from the ``object-oriented'' musical content to an operational musical process. As such, transformations between musical elements are emphasized, rather than the musical elements themselves. In the original framework of Lewin, the set of transformations often form a group, with a corresponding group action on a given set of musical objects. Within this framework, Klumpenhouwer networks (henceforth K-nets) \cite{Lewin_1990,Klumpenhouwer_1991,Klumpenhouwer_1998} have stressed the deep synergy between set-theoretical and transformational approaches thanks to their anchoring in both group and graph theory, as observed by many scholars \cite{Nolan_2007}. We recall that a K-net is informally defined as a labelled graph, wherein the labels of the vertices belong to the set of pitch classes, and each arrow is labelled with a transformation that maps the pitch class at the source vertex to the pitch class at the target vertex. Klumpenhouwer networks allow one to conveniently visualize at once the musical elements of a set and the specific transformations between them. This notion has been later formalized in a more categorical setting, first as limits of diagrams within the framework of denotators \cite{Mazzola_2006}, and later as a special case of a categorical construction called poly-Klumpenhouwer networks (PK-nets) \cite{PopoffMCM2015,Popoff2016}.

The goal of this paper is to propose a new groupoid-based approach to transformational music theory, in which transformations of PK-nets are considered rather than ordinary sets of musical objects. The first section shows how groupoids of musical transformations can be constructed, and how they can be applied to post-tonal music analysis. The second section shows how groupoids are linked to wreath products through the notion groupoid bisections. Wreath products feature prominently in transformational music theory \cite{Hook_2002}. The work presented in this paper thus shows how groupoids of musical transformations can be related to the more traditional group-based approach of transformational music theory.

We assume that the reader is familiar with the basic notions of transformational music analysis, in particular with the so-called $T\text{/}I$ group and its action on the set of the twelve pitch classes (see \cite{Fiore_2011,Fiore_2013} for additional information).

\subsection{Introduction to PK-Nets}

The groupoid-based approach to transformational music theory presented in this paper stems from the constitutive elements of poly-Klumpenhouwer networks which have been introduced previously \cite{PopoffMCM2015,Popoff2016}. We recall the categorical definition of a PK-net, which generalizes the original notion of K-nets in various ways.

\begin{definition}
Let $\mathbf{C}$ be a category, and $S$ a functor from $\mathbf{C}$ to the category $\mathbf{Sets}$ of (small) sets. Let $\Delta$ be a small category and $R$ a functor from $\Delta$ to $\mathbf{Sets}$ with non-empy values. A {\it PK-net of form $R$ and of support $S$} is a 4-tuple $(R,S,F,\phi)$, in which $F$ is a functor from $\Delta$ to $\mathbf{C}$, and $\phi$ is a natural transformation from $R$ to $SF$.
\end{definition}

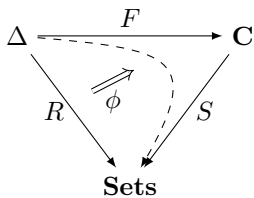
\begin{figure}
	\centering
	\begin{tikzpicture}
		\node (A) at (0,0) {$\Delta$};
		\node (B) at (3,0) {$\mathbf{C}$};
		\node (C) at (1.5,-2) {$\mathbf{Sets}$};
		\draw[->,>=latex] (A) -- (B) node[above,midway]{$F$};
		\draw[->,>=latex] (B) -- (C) node[right,midway] (D) {$S$} ;
		\draw[->,>=latex] (A) -- (C) node[left,midway]  (E) {$R$};
		\draw[->,>=latex,dashed] (A) to[out=-5,in=60,looseness=2.0] node[above,midway](F){} (C) ;
		\draw[-Implies,double distance=1.5pt,shorten >=8pt,shorten <=8pt] (E) to node[below,midway] {$\phi$} (F);
	\end{tikzpicture}
\caption{Diagrammatic representation of a PK-net $(R,S,F,\phi)$.}
\label{fig:PKNet_diagram}
\end{figure}

A PK-net can be represented by the diagram of Figure \ref{fig:PKNet_diagram}. Among the constitutive elements of a PK-net, the category $\mathbf{C}$ and the functor $S \colon \mathbf{C} \to \mathbf{Sets}$ represent the context of the analysis. Traditional transformational music theory commonly relies on a group acting on a given set of objects: the most well-known examples are the $T$/$I$ group acting on the set of the twelve pitch classes, the same $T$/$I$ group acting simply transitively on the set of the 24 major and minor triads, or the $PLR$ group acting simply transitively on the same set, to name a few examples. From a categorical point of view, the data of a group and its action on a set is equivalent to the data of a functor from a single-object category with invertible morphisms to the category of sets. However, this situation can be further generalized by considering any category $\mathbf{C}$ along with a functor $S \colon \mathbf{C} \to \mathbf{Sets}$. The morphisms of the category $\mathbf{C}$ are therefore the musical transformations of interest.  The category $\Delta$ serves as the abstract skeleton of the PK-net: as such, its objects and morphisms are abstract entities, which are labelled by the functor $F$ from $\Delta$ to the category $\mathbf{C}$. The objects of $\Delta$ do not represent the actual musical elements of a PK-net: these are introduced by the functor $R$ from $\Delta$ to $\mathbf{Sets}$. This functor sends each object of $\Delta$ to an actual set, which may contain more than a single element, and whose elements abstractly represent the musical objects of study. However, these elements are not yet labelled. In the same way the morphisms of $\Delta$ represent abstract relationships which are given a concrete meaning by the functor $F$, the elements in the images of $R$ are given a label in the images of $S$ through the natural transformation $\phi$. The elements in the image of $S$ represent musical entities on which the category $\mathbf{C}$ acts, and one would therefore need a way to connect the elements in the image of $R$ with those in the image of $S$. However, one cannot simply consider a collection of functions between the images of $R$ and the images of $S$ in order to label the musical objects in the PK-net. Indeed, one must make sure that two elements in the images of $R$ which are related by a function $R(f)$ (with $f$ being a morphism of $\Delta$) actually correspond to two elements in the images of $S$ related by the function $SF(f)$. The purpose of the natural transformation $\phi$ is thus to ensure the coherence of the whole diagram.

\subsection{Reinterpreting the constitutive elements of a PK-Net}

Figure \ref{fig:KNet} shows a basic Klumpenhouwer network describing a $C$ major triad. The arrows of this network are labelled with specific transformations in the $T$/$I$ group which express the fact that the major triad is made up of two notes, $E$ and $G$, separated from the tonic $C$ by a major third and a fifth respectively.

\begin{figure}
\begin{center}
\begin{tikzpicture}[scale=1]
	\node (A) at (0,0) {$C$};
	\node (B) at (1,-1.3) {$E$};
	\node (C) at (0,-2.3) {$G$};
	\draw[->,>=latex] (A) to node[shift={(0.3,0.15)}]{$T_4$} (B) ;
	\draw[->,>=latex] (B) to node[shift={(0.3,-0.15)}]{$T_3$} (C) ;
	\draw[->,>=latex] (A) to node[left,midway]{$T_7$} (C) ;
\end{tikzpicture}
\end{center}
\caption{A Klumpenhouwer network (K-net) describing a major triad. The arrows are labelled with specific transformations in the $T$/$I$ group relating the pitch classes in their domain and codomain.}
\label{fig:KNet}
\end{figure}
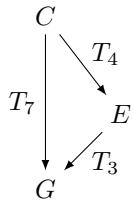

In the terminology of PK-Nets, this network is described by
\begin{enumerate}
\item{a category $\Delta_3$ with three objects $X$, $Y$, and $Z$, and three non-trivial morphisms $f \colon X \to Y$, $g \colon Y \to Z$, and $g \circ f \colon X \to Z$ between them, and}
\item{a category $\mathbf{C}$ taken here to be the $T$/$I$ group, with its usual action $S \colon T\text{/}I \to \mathbf{Sets}$ on the set of the twelve pitch-classes, and}
\item{a functor $F \colon \Delta_3 \to \mathbf{C}$ such that $F(f)=T_4$, $F(g)=T_3$, and $F(g \circ f)=T_7$, and}
\item{a functor $R \colon \Delta_3 \to \mathbf{Sets}$ sending each object of $\Delta_3$ to a singleton, and a natural transformation $\phi$ sending these singletons to the appropriate pitch-classes in the image of $\mathbf{C}$ by $S$.}
\end{enumerate}

Upon examination of these constitutive elements, it readily appears that the nature of this chord, here a major triad, is entirely determined by the category $\Delta_3$ and the functor $F \colon \Delta_3 \to \mathbf{C}$. The images of the morphisms of $\Delta_3$ by $F$ reflect the fact that a major triad is made up of a major third with a minor third stacked above it. This observation is not specific to major triads. Consider for example the chord $\{D, E, G\}$, which is a representative of the set class $\left[0,2,5\right]$. This chord may be described by a PK-Net in which $\Delta_3$ is the same as above, and in which we consider a different functor $F' \colon \Delta_3 \to \mathbf{C}$ such that $F(f)=T_2$, $F(g)=T_5$, and $F(g \circ f)=T_7$.

One may go further by considering pitch-class sets which are not necessarily transpositionnally related. Figure \ref{fig:webernop11} shows an excerpt of Webern's \textit{Three Little Pieces for Cello and Piano, Op. 11/2}, at bars 4-5, along with a PK-net interpretation of each three-note segment. These three-note segments are clearly not related by transposition, yet the corresponding represented networks share the same functor $\Delta_3 \to \mathbf{C}$.

By abstracting this observation, one can consider that these three-note pitch-class sets belong to the same \textit{generalized chord type}, which is defined by this particular functor $\Delta_3 \to \mathbf{C}$. The main point of this paper is to generalize further these observations by considering functors $F \colon \Delta \to \mathbf{C}$ as \textit{generalized musical classes}.

\begin{figure}[t!]
\begin{center}
\subfigure[]{
\includegraphics[scale=0.8]{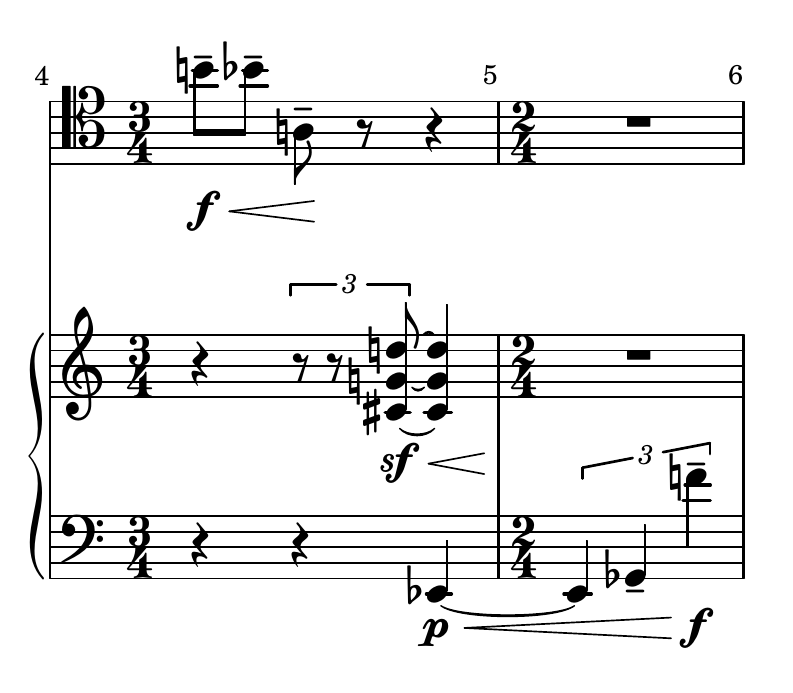}
\label{subfig:webernop11_1}
}
\subfigure[]{
\begin{tikzpicture}
	\node (A) at (0,0) {$A$};
	\node (B) at (1,-1) {$B$};
	\node (C) at (0,-2) {$B_\flat$};
	\draw[->,>=latex] (A) -- (B) node[above,midway]{$I_8$};
	\draw[->,>=latex] (B) -- (C) node[right,midway]{$I_{9}$};
	\draw[->,>=latex] (A) -- (C) node[left,midway]{$T_1$};
	
	\node (D) at (4,0) {$C_\sharp$};
	\node (E) at (5,-1) {$G$};
	\node (F) at (4,-2) {$D$};
	\draw[->,>=latex] (D) -- (E) node[above,midway]{$I_8$};
	\draw[->,>=latex] (E) -- (F) node[right,midway]{$I_{9}$};
	\draw[->,>=latex] (D) -- (F) node[left,midway]{$T_1$};

	\node (G) at (8,0) {$F$};
	\node (H) at (9,-1) {$E_\flat$};
	\node (I) at (8,-2) {$F_\sharp$};
	\draw[->,>=latex] (G) -- (H) node[above,midway]{$I_8$};
	\draw[->,>=latex] (H) -- (I) node[right,midway]{$I_{9}$};
	\draw[->,>=latex] (G) -- (I) node[left,midway]{$T_1$};
\end{tikzpicture}
\label{subfig:webernop11_2}
}
\end{center}
\caption{\subref{subfig:webernop11_1} Webern, Op. 11/2, bars 4-5. \subref{subfig:webernop11_2} PK-nets corresponding to each of three-note segment of \subref{subfig:webernop11_1}.}
\label{fig:webernop11}
\end{figure}

\begin{definition}
Let $\mathbf{C}$ be a category, and $\Delta$ be a small category. A {\it generalized musical class} is a functor $F \colon \Delta \to\mathbf{C}$. When $\mathbf{C}$ is the particular case of the group $T$/$I$, a functor $F \colon \Delta \to\mathbf{C}$ is said to be a {\it generalized chord class}.
\end{definition}

As is well-known in category theory, functors $F \colon \Delta \to\mathbf{C}$ form a category, known as \textit{the category of functors $\mathbf{C}^{\Delta}$}.

\begin{definition}
The category of functors $\mathbf{C}^{\Delta}$ has
\begin{enumerate}
\item{functors $F \colon \Delta \to\mathbf{C}$ as objects, and}
\item{natural transformations $\eta \colon F \to F'$ between functors $F \colon \Delta \to \mathbf{C}$ and $F' \colon \Delta \to \mathbf{C}$ as morphisms.}
\end{enumerate}
\end{definition}

These natural transformations can be seen as generalized musical transformations between the corresponding generalized musical classes. The additional data of functors $R$ and $S$, and of a natural transformation $\phi \colon R \to SF$, leads to individual musical sets derived from the generalized musical class $F \colon \Delta \to \mathbf{C}$.
The purpose of this paper is to investigate the structure of $\mathbf{C}^{\Delta}$ and its action on $\mathbf{Sets}$, and to relate these constructions with known group-theoretical results in transformational music analysis.

One may notice that different categories $\Delta$ and functors $F \colon \Delta \to \mathbf{C}$ may describe the same musical sets. For example, the major triad of Figure \ref{fig:KNet} may also be described using a category $\Gamma$ with three objects $X$, $Y$, and $Z$, and only two non-trivial morphisms $f \colon X \to Y$, $g \colon X \to Z$, between them, and a functor $F'' \colon \Gamma \to \mathbf{C}=T/I$ sending $f$ to $T_4$ and $g$ to $T_7$. This alternative description, shown in Figure \ref{fig:KNet_alt}, focuses on the major third and the fifth without referencing the minor third. Instead of being a limitation, this possibility allows for various transformations between chord types to be examined, as will be seen in the rest of the paper. 

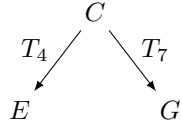
\begin{figure}
\begin{center}
\begin{tikzpicture}[scale=1]
	\node (A) at (0,0) {$C$};
	\node (B) at (1,-1.3) {$G$};
	\node (C) at (-1,-1.3) {$E$};
	\draw[->,>=latex] (A) to node[shift={(0.3,0.15)}]{$T_7$} (B) ;
	\draw[->,>=latex] (A) to node[shift={(-0.3,0.15)}]{$T_4$} (C) ;
\end{tikzpicture}
\end{center}
\caption{An alternative Klumpenhouwer network describing a major triad.}
\label{fig:KNet_alt}
\end{figure}

As stated in the previous subsection, the category $\mathbf{C}$ is often a group in musical applications. The following proposition establishes the structure of the category of functors $\mathbf{C}^{\Delta}$ when $\Delta$ is a poset with a bottom element $O$ and $\mathbf{C}$ is a group $\mathbf{G}$ considered as a single-object category.

\begin{proposition}\label{proposition:GDeltaStruct}
Let $\Delta$ be a poset with a bottom element $O$ and $\mathbf{G}$ be a group considered as a category. Then
\begin{enumerate}
\item{the category of functors $\mathbf{G}^{\Delta}$ is a groupoid, and}
\item{for any two objects $F$ and $F'$ of $\mathbf{G}^{\Delta}$ the hom-set $\text{Hom}(F,F')$ can be bijectively identified with the set of elements of $\mathbf{G}$.}
\end{enumerate}
\begin{proof}
Given two objects $F$ and $F'$ of $\mathbf{G}^{\Delta}$, i.e. two functors $F \colon \Delta \to \mathbf{G}$ and $F' \colon \Delta \to \mathbf{G}$, any natural transformation $\eta \colon F \to F'$ between them is invertible, since the components of $\eta$ are invertible morphisms of $\mathbf{G}$. Thus, the category of functors $\mathbf{G}^{\Delta}$ is a groupoid.
Since $\Delta$ is a poset with a bottom element $O$, the natural transformation $\eta$ is entirely determined by the component $\eta_O$, which can be freely chosen in $\mathbf{G}$. Thus the hom-set $\text{Hom}(F,F')$ can be bijectively identified with the set of elements of $\mathbf{G}$.
\end{proof}
\end{proposition}

Since in such a case the elements of the hom-set $\text{Hom}(F,F')$ can be uniquely identified with the elements of $\mathbf{G}$, we will use the notation ${}^{FF'}g$, with $g \in \mathbf{G}$, to designate an element of $\text{Hom}(F,F')$ in $\mathbf{G}^{\Delta}$.

We now consider a functor $R \colon \Delta \to \mathbf{Sets}$ and a functor $S \colon \mathbf{G} \to \mathbf{Sets}$. For any object $F$ of $\mathbf{G}^{\Delta}$, we denote by $N_F$ the set of all natural transformations $\phi \colon R \to SF$.

\begin{proposition}
There exists a canonical functor $P_{R,S} \colon \mathbf{G}^{\Delta} \to \mathbf{Sets}$ such that for any object $F$ of $\mathbf{G}^{\Delta}$, $P_{R,S}(F)=N_F$.
\begin{proof}
We consider the functor $P_{R,S} \colon \mathbf{G}^{\Delta} \to \mathbf{Sets}$, which sends each object $F$ of $\mathbf{G}^{\Delta}$ to the set $N_F$. Then, given two objects $F$ and $F'$ of $\mathbf{G}^{\Delta}$ and any morphism $\eta \colon F \to F'$ between them, we can construct the image of $\eta$ by $P_{R,S}$ as the map $P_{R,S}(\eta) \colon N_F \to N_{F'}$ sending a natural transformation $\phi$ of $N_F$ to the natural transformation $S\eta \circ \phi$ of $N_{F'}$.
\end{proof}
\end{proposition}

\subsection{Transformations of generalized chord classes}

We now consider the specific case where the category $\mathbf{G}$ corresponds to the $T\text{/}I$ group. We wish here to give examples of transformations of generalized chord classes, in the particular case where $\Delta$ is the category $\Gamma$ introduced above, and with specific functors $U \colon \Gamma \to \mathbf{G}$ and $V \colon \Gamma \to \mathbf{G}$. Our goal is to detail the structure of the hom-sets $\text{Hom}(U,U)$ and $\text{Hom}(U,V)$. 

We consider the following objects of ${(T\text{/}I)}^{\Gamma}$:

\begin{itemize}
\item{the functor $U \colon \Gamma \to T/I$ sending $f$ to $T_4$ and $g$ to $T_{7}$, and}
\item{the functor $V \colon \Gamma \to T/I$ sending $f$ to $T_2$ and $g$ to $T_{5}$.}
\end{itemize}

These functors model the set classes of prime form [0,4,7] (major triad) and [0,2,5]. We now consider the hom-set $\text{Hom}(U,U)$ in ${(T\text{/}I)}^{\Gamma}$. Let $\eta \colon U \to U$ be a natural transformation. It is uniquely determined by the component $\eta_X$, which is an element of $T\text{/}I$. We can derive the components $\eta_Y$ and $\eta_Z$ depending on $\eta_X$.

\begin{itemize}
\item{If $\eta_X=T_p$, with $p$ in $\{0 \ldots 11\}$, then we must have $\eta_Y T_4 = T_4 T_p$, and $\eta_Z T_{7} = T_{7} T_p$. This leads to $\eta_Y = \eta_Z = T_{p}$.}
\item{If $\eta_X=I_p$, with $p$ in $\{0 \ldots 11\}$, then we must have $\eta_Y T_4 = T_4 I_p$, and $\eta_Z T_{7} = T_{7} I_p$. This leads to $\eta_Y = I_{p+8}$, and $\eta_Z = I_{p+2}$.}
\end{itemize}

The first type of transformation ${}^{UU}T_p$ is well known as it is simply the usual transposition of the set class. The second type of transformation ${}^{UU}I_p$ can be seen as a ``generalized'' inversion. Unlike the known action on triads of the inversions of the $T\text{/}I$ group, wherein the same inversion $I_p$ operates on every pitch class of the chord, the ``generalized'' inversion ${}^{UU}I_p$ has different components for each object of $\Gamma$. The Figure \ref{fig:chord_transformation_1} illustrates the action of the morphism ${}^{UU}I_8$ on the PK-net representing the $F$ major chord, resulting in the $E_\flat$ major chord (the constitutive elements of the PK-net, namely the functor $R$, $S$, and the natural transformation $\phi$, have been omitted in this Figure for clarity).

\begin{figure}
\begin{center}
\subfigure[]{
\begin{tikzpicture}[scale=1]
	\node (A0) at (0,0) {$F$};
	\node (B0) at (-0.5,-1.3) {$A$};
	\node (C0) at (0.5,-1.3) {$C$};
	\draw[->,>=latex] (A0) to node[shift={(-0.3,0.15)}]{$T_4$} (B0) ;
	\draw[->,>=latex] (A0) to node[shift={(0.3,0.15)}]{$T_{7}$} (C0) ;

	\node (A1) at (3,0) {$E\flat$};
	\node (B1) at (2.5,-1.3) {$G$};
	\node (C1) at (3.5,-1.3) {$B\flat$};
	\draw[->,>=latex] (A1) to node[shift={(-0.3,0.15)}]{$T_4$} (B1) ;
	\draw[->,>=latex] (A1) to node[shift={(0.3,0.15)}]{$T_{7}$} (C1) ;
	
	\draw[->,>=latex, red] (1,-0.6)  to node[above]{${}^{UU}I_{8}$} (2,-0.6) ;

	\draw[->,>=latex, violet] (A0) to[bend left=20] node[above]{$I_{8}$} (A1) ;
	\draw[->,>=latex, violet] (B0) to[bend right=20] node[below]{$I_{4}$} (B1) ;
	\draw[->,>=latex, violet] (C0) to[bend right=20] node[below]{$I_{10}$} (C1) ;	
\end{tikzpicture}
\label{subfig:chord_transformation_1_1}
}

\vspace{0.7cm}
\subfigure[]{
\begin{tikzpicture}[scale=0.9] 
%% red, green, blue, cyan , magenta, yellow, black, gray, darkgray, lightgray, brown, lime, olive, orange, pink, purple, teal, violet and white.
	\draw (-5,0) circle[radius=2.0]; 
	\foreach \i in {0,...,11} {
		\node[circle,draw=black,fill=black, fill opacity = 0.05, inner sep=1pt, minimum size=3pt] (A\i) at ({-5.0+2.0*sin(\i*30)},{2.0*cos(\i*30)}) {.};
	}

	% Inversion nodes
	\foreach \i in {0,...,23} {
		\node (INV\i) at ({-5.0+3.0*sin(\i*15)},{3.0*cos(\i*15)}) {};
	}

	\foreach \i / \name in {0/$C$,1/$C_\sharp$,2/$D$,3/$E_\flat$,4/$E$,5/$F$,6/$F_\sharp$,7/$G$,8/$G_\sharp$,9/$A$,10/$B_\flat$,11/$B$} {
		\node (N\i) at ({-5.0+2.5*sin(\i*30)},{2.5*cos(\i*30)}) {\name};
	}
	\draw[ draw=black, fill=cyan, fill opacity=0.2] (A5.center) -- (A9.center) -- (A0.center) --cycle;
	\draw[ -, dashed, line width=1.5, color=violet] (INV8) to (INV20); 
	\draw[ -, dashed, line width=1.5, color=violet] (INV4) to (INV16); 
	\draw[ -, dashed, line width=1.5, color=violet] (INV10) to (INV22); 
	\draw[ latex-latex, line width=1.0, color=violet] (A5.center) to (A3.center);
	\draw[ latex-latex, line width=1.0, color=violet] (A0.center) to (A10.center);
	\draw[ latex-latex, line width=1.0, color=violet] (A9.center) to (A7.center);

	\draw[->,>=latex, color=black,line width=2,] (-1.8,0.0) to node[above,midway]{${}^{UU}I_8$} (0.2,0) ;

	\draw (3,0) circle[radius=2.0]; 
	\foreach \i in {0,...,11} {
		\node[circle,draw=black,fill=black, fill opacity = 0.05, inner sep=1pt, minimum size=3pt] (C\i) at ({3+2.0*sin(\i*30)},{2.0*cos(\i*30)}) {.};
	}
	\foreach \i in {0,...,11} {
		\node (D\i) at ({3+2.5*sin(\i*30)},{2.5*cos(\i*30)}) {};
	}
	\draw[ draw=black, fill=cyan, fill opacity=0.2] (C3.center) -- (C7.center) -- (C10.center) --cycle;
\end{tikzpicture}
\label{subfig:chord_transformation_1_2}
}
\end{center}
\caption{\subref{subfig:chord_transformation_1_1} Action of the morphism ${}^{UU}I_8$ of ${(T\text{/}I)}^{\Gamma}$ on the PK-net representing the $F$ major chord, resulting in the $E_\flat$ major chord. The constitutive elements of the PK-nets (the functor $R$, $S$, and the natural transformation $\phi$) have been omitted here for clarity. \subref{subfig:chord_transformation_1_2} Graphical representation of the inversion components of the morphism ${}^{UU}I_8$, and their action on the individual pitch classes of the $F$ major chord.}
\label{fig:chord_transformation_1}
\end{figure}
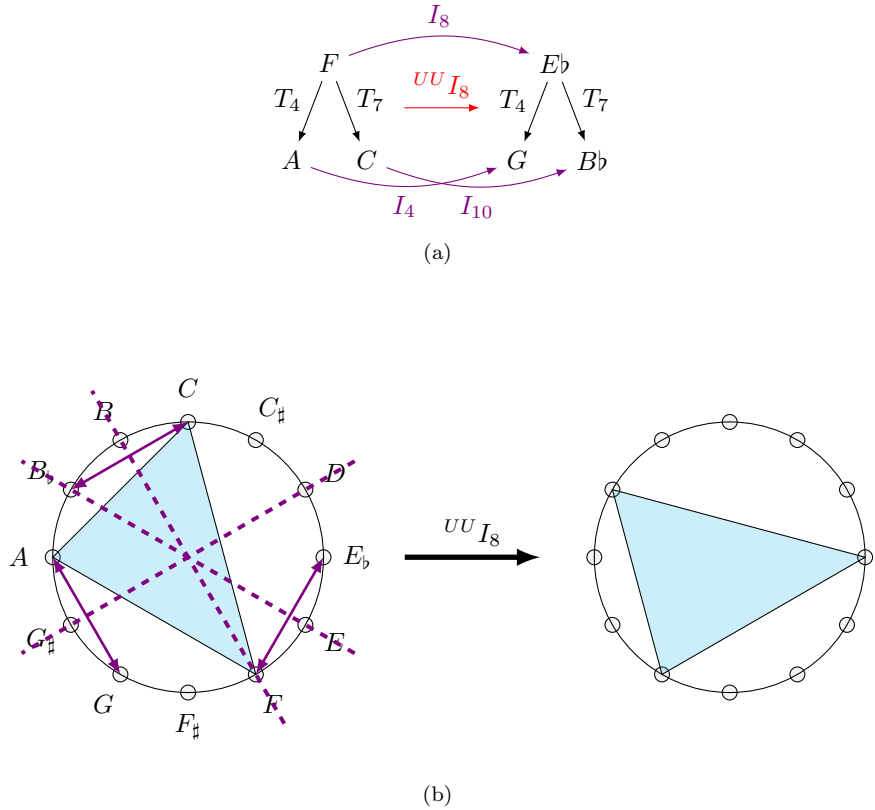

\begin{figure}
\begin{center}
\subfigure[]{
\begin{tikzpicture}[scale=1]
	\node (A0) at (0,0) {$F$};
	\node (B0) at (-0.5,-1.3) {$A$};
	\node (C0) at (0.5,-1.3) {$C$};
	\draw[->,>=latex] (A0) to node[shift={(-0.3,0.15)}]{$T_4$} (B0) ;
	\draw[->,>=latex] (A0) to node[shift={(0.3,0.15)}]{$T_{7}$} (C0) ;

	\node (A1) at (3,0) {$A\flat$};
	\node (B1) at (2.5,-1.3) {$B\flat$};
	\node (C1) at (3.5,-1.3) {$C\sharp$};
	\draw[->,>=latex] (A1) to node[shift={(-0.3,0.15)}]{$T_2$} (B1) ;
	\draw[->,>=latex] (A1) to node[shift={(0.3,0.15)}]{$T_{5}$} (C1) ;
	
	\draw[->,>=latex, red] (1,-0.6)  to node[above]{${}^{UV}T_{3}$} (2,-0.6) ;

	\draw[->,>=latex, violet] (A0) to[bend left=20] node[above]{$T_{3}$} (A1) ;
	\draw[->,>=latex, violet] (B0) to[bend right=20] node[below]{$T_{1}$} (B1) ;
	\draw[->,>=latex, violet] (C0) to[bend right=20] node[below]{$T_{1}$} (C1) ;
	
	\node (A2) at (6,0) {$F$};
	\node (B2) at (5.5,-1.3) {$A$};
	\node (C2) at (6.5,-1.3) {$C$};
	\draw[->,>=latex] (A2) to node[shift={(-0.3,0.15)}]{$T_4$} (B2) ;
	\draw[->,>=latex] (A2) to node[shift={(0.3,0.15)}]{$T_{7}$} (C2) ;

	\node (A3) at (9,0) {$A\flat$};
	\node (B3) at (8.5,-1.3) {$B\flat$};
	\node (C3) at (9.5,-1.3) {$C\sharp$};
	\draw[->,>=latex] (A3) to node[shift={(-0.3,0.15)}]{$T_2$} (B3) ;
	\draw[->,>=latex] (A3) to node[shift={(0.3,0.15)}]{$T_{5}$} (C3) ;
	
	\draw[->,>=latex, red] (7,-0.6)  to node[above]{${}^{UV}I_{1}$} (8,-0.6) ;

	\draw[->,>=latex, violet] (A2) to[bend left=20] node[above]{$I_{1}$} (A3) ;
	\draw[->,>=latex, violet] (B2) to[bend right=20] node[below]{$I_{7}$} (B3) ;
	\draw[->,>=latex, violet] (C2) to[bend right=20] node[below]{$I_{1}$} (C3) ;	
	
\end{tikzpicture}
\label{subfig:chord_transformation_2_1}
}
\subfigure[]{
\begin{tikzpicture}[scale=0.9]
%% red, green, blue, cyan , magenta, yellow, black, gray, darkgray, lightgray, brown, lime, olive, orange, pink, purple, teal, violet and white.
	\draw (-5,0) circle[radius=2.0]; 
	\foreach \i in {0,...,11} {
		\node[circle,draw=black,fill=black, fill opacity = 0.05, inner sep=1pt, minimum size=3pt] (A\i) at ({-5.0+2.0*sin(\i*30)},{2.0*cos(\i*30)}) {.};
	}

	% Inversion nodes
	\foreach \i in {0,...,11} {
		\node (TRANSP\i) at ({-5.0+3.0*sin(\i*30)},{3.0*cos(\i*30)}) {};
	}

	\foreach \i / \name in {0/$C$,1/$C_\sharp$,2/$D$,3/$E_\flat$,4/$E$,5/$F$,6/$F_\sharp$,7/$G$,8/$G_\sharp$,9/$A$,10/$B_\flat$,11/$B$} {
		\node (N\i) at ({-5.0+2.5*sin(\i*30)},{2.5*cos(\i*30)}) {\name};
	}
	\draw[ draw=black, fill=cyan, fill opacity=0.2] (A5.center) -- (A9.center) -- (A0.center) --cycle;
	\draw[ -latex, line width=1.0, color=violet] (TRANSP5) to[bend left=40] (TRANSP8);
	\draw[ -latex, line width=1.0, color=violet] (TRANSP9) to[bend left=20] (TRANSP10);
	\draw[ -latex, line width=1.0, color=violet] (TRANSP0) to[bend left=20] (TRANSP1);

	\draw[->,>=latex, color=black,line width=2,] (-1.8,0.0) to node[above,midway]{${}^{UV}T_3$} (0.2,0) ;

	\draw (3,0) circle[radius=2.0]; 
	\foreach \i in {0,...,11} {
		\node[circle,draw=black,fill=black, fill opacity = 0.05, inner sep=1pt, minimum size=3pt] (C\i) at ({3+2.0*sin(\i*30)},{2.0*cos(\i*30)}) {.};
	}
	\foreach \i in {0,...,11} {
		\node (D\i) at ({3+2.5*sin(\i*30)},{2.5*cos(\i*30)}) {};
	}
	\draw[ draw=black, fill=cyan, fill opacity=0.2] (C8.center) -- (C10.center) -- (C1.center) --cycle;
\end{tikzpicture}
\label{subfig:chord_transformation_2_2}
}
\subfigure[]{
\begin{tikzpicture}[scale=0.9]
%% red, green, blue, cyan , magenta, yellow, black, gray, darkgray, lightgray, brown, lime, olive, orange, pink, purple, teal, violet and white.
	\draw (-5,0) circle[radius=2.0]; 
	\foreach \i in {0,...,11} {
		\node[circle,draw=black,fill=black, fill opacity = 0.05, inner sep=1pt, minimum size=3pt] (A\i) at ({-5.0+2.0*sin(\i*30)},{2.0*cos(\i*30)}) {.};
	}

	% Inversion nodes
	\foreach \i in {0,...,23} {
		\node (INV\i) at ({-5.0+3.0*sin(\i*15)},{3.0*cos(\i*15)}) {};
	}

	\foreach \i / \name in {0/$C$,1/$C_\sharp$,2/$D$,3/$E_\flat$,4/$E$,5/$F$,6/$F_\sharp$,7/$G$,8/$G_\sharp$,9/$A$,10/$B_\flat$,11/$B$} {
		\node (N\i) at ({-5.0+2.5*sin(\i*30)},{2.5*cos(\i*30)}) {\name};
	}
	\draw[ draw=black, fill=cyan, fill opacity=0.2] (A5.center) -- (A9.center) -- (A0.center) --cycle;
	\draw[ -, dashed, line width=1.5, color=violet] (INV1) to (INV13); 
	\draw[ -, dashed, line width=1.5, color=violet] (INV7) to (INV19); 
	\draw[ latex-latex, line width=1.0, color=violet] (A5.center) to (A8.center);
	\draw[ latex-latex, line width=1.0, color=violet] (A0.center) to (A1.center);
	\draw[ latex-latex, line width=1.0, color=violet] (A9.center) to (A10.center);

	\draw[->,>=latex, color=black,line width=2,] (-1.8,0.0) to node[above,midway]{${}^{UV}I_1$} (0.2,0) ;

	\draw (3,0) circle[radius=2.0]; 
	\foreach \i in {0,...,11} {
		\node[circle,draw=black,fill=black, fill opacity = 0.05, inner sep=1pt, minimum size=3pt] (C\i) at ({3+2.0*sin(\i*30)},{2.0*cos(\i*30)}) {.};
	}
	\foreach \i in {0,...,11} {
		\node (D\i) at ({3+2.5*sin(\i*30)},{2.5*cos(\i*30)}) {};
	}
	\draw[ draw=black, fill=cyan, fill opacity=0.2] (C8.center) -- (C10.center) -- (C1.center) --cycle;
\end{tikzpicture}
\label{subfig:chord_transformation_2_3}
}
\end{center}
\caption{\subref{subfig:chord_transformation_2_1} Action of the morphisms ${}^{UV}T_3$ and ${}^{UV}I_1$ of ${(T\text{/}I)}^{\Gamma}$ on the PK-net representing the $F$ major chord, resulting in the \{$A\flat$, $B\flat$,$C\sharp$\} chord. The constitutive elements of the PK-nets (the functor $R$, $S$, and the natural transformation $\phi$) have been omitted here for clarity. \subref{subfig:chord_transformation_2_2} and \subref{subfig:chord_transformation_2_3} Graphical representation of the transposition and inversion components of the morphisms ${}^{UV}T_3$ and ${}^{UV}I_1$, and their action on the individual pitch classes of the $F$ major chord.}
\label{fig:chord_transformation_2}
\end{figure}

In the same way, we can consider the hom-set $\text{Hom}(U,V)$ in ${(T\text{/}I)}^{\Gamma}$. Let $\eta \colon U \to V$ be a natural transformation. It is uniquely determined by the component $\eta_X$, which is an element of $T$/$I$. We can derive the components $\eta_Y$ and $\eta_Z$ depending on $\eta_X$.

\begin{itemize}
\item{If $\eta_X=T_p$, with $p$ in $\{0 \ldots 11\}$, then we must have $\eta_Y T_4 = T_2 T_p$, and $\eta_Z T_{7} = T_{5} T_p$. This leads to $\eta_Y = T_{p+10}$, and $\eta_Z = T_{p+10}$.}
\item{If $\eta_X=I_p$, with $p$ in $\{0 \ldots 11\}$, then we must have $\eta_Y T_4 = T_2 I_p$, and $\eta_Z T_{7} = T_{5} I_p$. This leads to $\eta_Y = I_{p+6}$, and $\eta_Z = I_{p}$.}
\end{itemize}

As before the morphisms of $\text{Hom}(U,V)$ have different components for each object of $\Gamma$. The morphisms ${}^{UV}T_p$ can be considered similarly as before as ``generalized'' transpositions, and morphisms ${}^{UV}I_p$ as ``generalized'' inversions between objects $U$ and $V$. The Figure \ref{fig:chord_transformation_2} illustrates the action of these morphism on the PK-net representing the $F$ major chord, resulting in the \{$A\flat$, $B\flat$,$C\sharp$\} chord (the constitutive elements of the PK-net, namely the functor $R$, $S$, and the natural transformation $\phi$, have been omitted in this Figure for clarity).

\subsection{An application to Berg's Op. 5/2}

To illustrate the above concepts, we will focus on a small atonal example from Berg's \textit{Four pieces for clarinet and piano, Op. 5/2}. Figure \ref{fig:bergop5} shows a reduction of the piano right hand part at bars 5-6. To analyse this progression, we consider the group $\mathbf{G}=T\text{/}I$, the category $\Gamma$ described above, and the corresponding groupoid of functors ${(T\text{/}I)}^{\Gamma}$. In particular we consider the following objects of ${(T\text{/}I)}^{\Gamma}$:

\begin{itemize}
\item{the functor $U \colon \Gamma \to T\text{/}I$ sending $f$ to $I_3$ and $g$ to $I_{10}$,}
\item{the functor $U' \colon \Gamma \to T\text{/}I$ sending $f$ to $I_7$ and $g$ to $I_{3}$,}
\item{the functor $V \colon \Gamma \to T\text{/}I$ sending $f$ to $I_4$ and $g$ to $I_{10}$, and}
\item{the functor $W \colon \Gamma \to T\text{/}I$ sending $f$ to $I_8$ and $g$ to $I_{3}$.}
\end{itemize}

We also consider  the functor $R \colon \Gamma \to \mathbf{Sets}$ sending each object of $\Gamma$ to a singleton, and the functor $S \colon T\text{/}I \to \mathbf{Sets}$ given by the action of the $T$/$I$ group on the set of the twelve pitch-classes. It can then easily be checked that the first five chords of the progression of Figure \ref{fig:bergop5} are instances of PK-nets using $R$ and $S$, and whose functor from $\Gamma$ to $T\text{/}I$ is either $U$ or $V$, as shown in Figure \ref{subfig:KNet_Berg_1_1}. Similarly, the last four chords of the progression of Figure \ref{fig:bergop5} are instances of PK-nets whose functor from $\Gamma$ to $T\text{/}I$ is either $U'$ or $W$, as shown in Figure \ref{subfig:KNet_Berg_2_1}.

\begin{figure}
\centering
\includegraphics[scale=0.65]{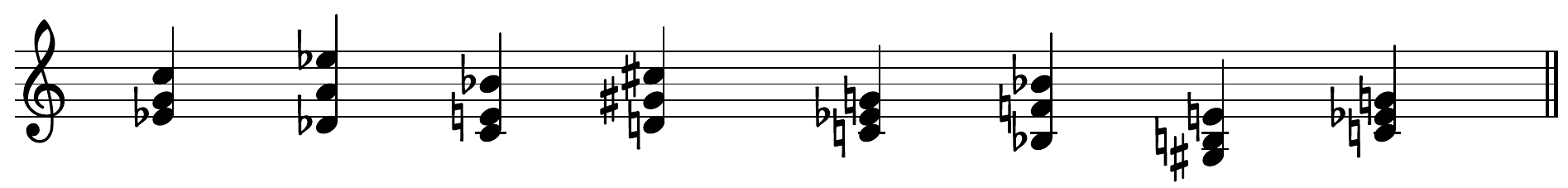}

\caption{Berg, Op. 5/2, reduction of the piano right hand part at bars 5-6.}
\label{fig:bergop5}
\end{figure}

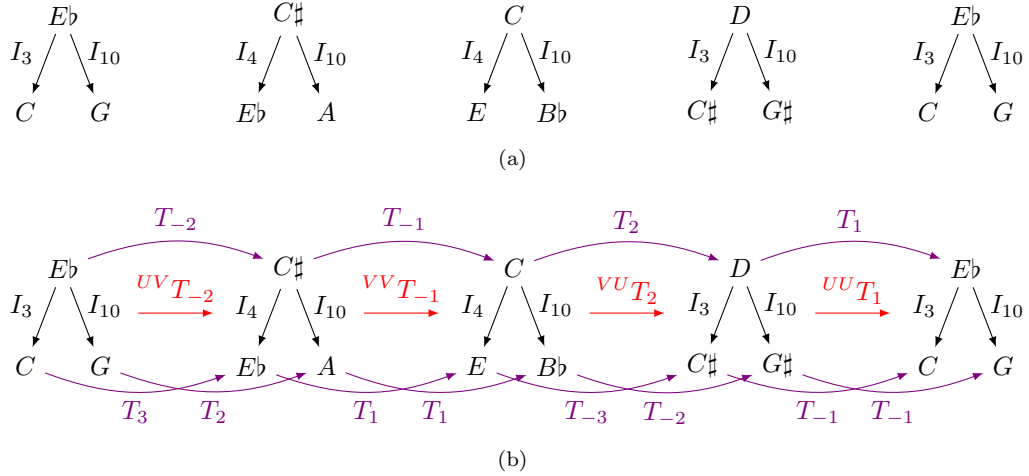
\begin{figure}
\begin{center}
\subfigure[]{
\begin{tikzpicture}[scale=1]
	\node (A0) at (0,0) {$E\flat$};
	\node (B0) at (-0.5,-1.3) {$C$};
	\node (C0) at (0.5,-1.3) {$G$};
	\draw[->,>=latex] (A0) to node[shift={(-0.3,0.15)}]{$I_3$} (B0) ;
	\draw[->,>=latex] (A0) to node[shift={(0.3,0.15)}]{$I_{10}$} (C0) ;

	\node (A1) at (3,0) {$C\sharp$};
	\node (B1) at (2.5,-1.3) {$E\flat$};
	\node (C1) at (3.5,-1.3) {$A$};
	\draw[->,>=latex] (A1) to node[shift={(-0.3,0.15)}]{$I_4$} (B1) ;
	\draw[->,>=latex] (A1) to node[shift={(0.3,0.15)}]{$I_{10}$} (C1) ;

	\node (A2) at (6,0) {$C$};
	\node (B2) at (5.5,-1.3) {$E$};
	\node (C2) at (6.5,-1.3) {$B\flat$};
	\draw[->,>=latex] (A2) to node[shift={(-0.3,0.15)}]{$I_4$} (B2) ;
	\draw[->,>=latex] (A2) to node[shift={(0.3,0.15)}]{$I_{10}$} (C2) ;

	\node (A3) at (9,0) {$D$};
	\node (B3) at (8.5,-1.3) {$C\sharp$};
	\node (C3) at (9.5,-1.3) {$G\sharp$};
	\draw[->,>=latex] (A3) to node[shift={(-0.3,0.15)}]{$I_3$} (B3) ;
	\draw[->,>=latex] (A3) to node[shift={(0.3,0.15)}]{$I_{10}$} (C3) ;

	\node (A4) at (12,0) {$E\flat$};
	\node (B4) at (11.5,-1.3) {$C$};
	\node (C4) at (12.5,-1.3) {$G$};
	\draw[->,>=latex] (A4) to node[shift={(-0.3,0.15)}]{$I_3$} (B4) ;
	\draw[->,>=latex] (A4) to node[shift={(0.3,0.15)}]{$I_{10}$} (C4) ;
\end{tikzpicture}
\label{subfig:KNet_Berg_1_1}
}
\subfigure[]{
\begin{tikzpicture}[scale=1]
	\node (A0) at (0,0) {$E\flat$};
	\node (B0) at (-0.5,-1.3) {$C$};
	\node (C0) at (0.5,-1.3) {$G$};
	\draw[->,>=latex] (A0) to node[shift={(-0.3,0.15)}]{$I_3$} (B0) ;
	\draw[->,>=latex] (A0) to node[shift={(0.3,0.15)}]{$I_{10}$} (C0) ;

	\node (A1) at (3,0) {$C\sharp$};
	\node (B1) at (2.5,-1.3) {$E\flat$};
	\node (C1) at (3.5,-1.3) {$A$};
	\draw[->,>=latex] (A1) to node[shift={(-0.3,0.15)}]{$I_4$} (B1) ;
	\draw[->,>=latex] (A1) to node[shift={(0.3,0.15)}]{$I_{10}$} (C1) ;

	\node (A2) at (6,0) {$C$};
	\node (B2) at (5.5,-1.3) {$E$};
	\node (C2) at (6.5,-1.3) {$B\flat$};
	\draw[->,>=latex] (A2) to node[shift={(-0.3,0.15)}]{$I_4$} (B2) ;
	\draw[->,>=latex] (A2) to node[shift={(0.3,0.15)}]{$I_{10}$} (C2) ;

	\node (A3) at (9,0) {$D$};
	\node (B3) at (8.5,-1.3) {$C\sharp$};
	\node (C3) at (9.5,-1.3) {$G\sharp$};
	\draw[->,>=latex] (A3) to node[shift={(-0.3,0.15)}]{$I_3$} (B3) ;
	\draw[->,>=latex] (A3) to node[shift={(0.3,0.15)}]{$I_{10}$} (C3) ;

	\node (A4) at (12,0) {$E\flat$};
	\node (B4) at (11.5,-1.3) {$C$};
	\node (C4) at (12.5,-1.3) {$G$};
	\draw[->,>=latex] (A4) to node[shift={(-0.3,0.15)}]{$I_3$} (B4) ;
	\draw[->,>=latex] (A4) to node[shift={(0.3,0.15)}]{$I_{10}$} (C4) ;
	
	\draw[->,>=latex, red] (1,-0.6)  to node[above]{${}^{UV}T_{-2}$} (2,-0.6) ;
	\draw[->,>=latex, red] (4,-0.6)  to node[above]{${}^{VV}T_{-1}$} (5,-0.6) ;
	\draw[->,>=latex, red] (7,-0.6)  to node[above]{${}^{VU}T_{2}$} (8,-0.6) ;
	\draw[->,>=latex, red] (10,-0.6)  to node[above]{${}^{UU}T_{1}$} (11,-0.6) ;

	\draw[->,>=latex, violet] (A0) to[bend left=20] node[above]{$T_{-2}$} (A1) ;
	\draw[->,>=latex, violet] (B0) to[bend right=20] node[below]{$T_{3}$} (B1) ;
	\draw[->,>=latex, violet] (C0) to[bend right=20] node[below]{$T_{2}$} (C1) ;

	\draw[->,>=latex, violet] (A1) to[bend left=20] node[above]{$T_{-1}$} (A2) ;
	\draw[->,>=latex, violet] (B1) to[bend right=20] node[below]{$T_{1}$} (B2) ;
	\draw[->,>=latex, violet] (C1) to[bend right=20] node[below]{$T_{1}$} (C2) ;

	\draw[->,>=latex, violet] (A2) to[bend left=20] node[above]{$T_{2}$} (A3) ;
	\draw[->,>=latex, violet] (B2) to[bend right=20] node[below]{$T_{-3}$} (B3) ;
	\draw[->,>=latex, violet] (C2) to[bend right=20] node[below]{$T_{-2}$} (C3) ;

	\draw[->,>=latex, violet] (A3) to[bend left=20] node[above]{$T_{1}$} (A4) ;
	\draw[->,>=latex, violet] (B3) to[bend right=20] node[below]{$T_{-1}$} (B4) ;
	\draw[->,>=latex, violet] (C3) to[bend right=20] node[below]{$T_{-1}$} (C4) ;	
\end{tikzpicture}
\label{subfig:KNet_Berg_1_2}
}
\end{center}
\caption{\subref{subfig:KNet_Berg_1_1} A PK-net interpretation of the first five chords of the progression of Figure \ref{fig:bergop5}. \subref{subfig:KNet_Berg_1_2} Analysis of the chord progression using morphisms of ${(T\text{/}I)}^{\Gamma}$.}
\label{fig:KNet_Berg_1}
\end{figure}

\begin{figure}
\begin{center}
\subfigure[]{
\begin{tikzpicture}[scale=1]
	\node (A0) at (0,0) {$C$};
	\node (B0) at (-0.5,-1.3) {$G$};
	\node (C0) at (0.5,-1.3) {$E\flat$};
	\draw[->,>=latex] (A0) to node[shift={(-0.3,0.15)}]{$I_7$} (B0) ;
	\draw[->,>=latex] (A0) to node[shift={(0.3,0.15)}]{$I_{3}$} (C0) ;

	\node (A1) at (3,0) {$B\flat$};
	\node (B1) at (2.5,-1.3) {$B\flat$};
	\node (C1) at (3.5,-1.3) {$F$};
	\draw[->,>=latex] (A1) to node[shift={(-0.3,0.15)}]{$I_8$} (B1) ;
	\draw[->,>=latex] (A1) to node[shift={(0.3,0.15)}]{$I_{3}$} (C1) ;

	\node (A2) at (6,0) {$B$};
	\node (B2) at (5.5,-1.3) {$G\sharp$};
	\node (C2) at (6.5,-1.3) {$E$};
	\draw[->,>=latex] (A2) to node[shift={(-0.3,0.15)}]{$I_7$} (B2) ;
	\draw[->,>=latex] (A2) to node[shift={(0.3,0.15)}]{$I_{3}$} (C2) ;

	\node (A3) at (9,0) {$C$};
	\node (B3) at (8.5,-1.3) {$G$};
	\node (C3) at (9.5,-1.3) {$E\flat$};
	\draw[->,>=latex] (A3) to node[shift={(-0.3,0.15)}]{$I_7$} (B3) ;
	\draw[->,>=latex] (A3) to node[shift={(0.3,0.15)}]{$I_{3}$} (C3) ;
\end{tikzpicture}
\label{subfig:KNet_Berg_2_1}
}
\subfigure[]{
\begin{tikzpicture}[scale=1]
	\node (A0) at (0,0) {$C$};
	\node (B0) at (-0.5,-1.3) {$G$};
	\node (C0) at (0.5,-1.3) {$E\flat$};
	\draw[->,>=latex] (A0) to node[shift={(-0.3,0.15)}]{$I_7$} (B0) ;
	\draw[->,>=latex] (A0) to node[shift={(0.3,0.15)}]{$I_{3}$} (C0) ;

	\node (A1) at (3,0) {$B\flat$};
	\node (B1) at (2.5,-1.3) {$B\flat$};
	\node (C1) at (3.5,-1.3) {$F$};
	\draw[->,>=latex] (A1) to node[shift={(-0.3,0.15)}]{$I_8$} (B1) ;
	\draw[->,>=latex] (A1) to node[shift={(0.3,0.15)}]{$I_{3}$} (C1) ;

	\node (A2) at (6,0) {$B$};
	\node (B2) at (5.5,-1.3) {$G\sharp$};
	\node (C2) at (6.5,-1.3) {$E$};
	\draw[->,>=latex] (A2) to node[shift={(-0.3,0.15)}]{$I_7$} (B2) ;
	\draw[->,>=latex] (A2) to node[shift={(0.3,0.15)}]{$I_{3}$} (C2) ;

	\node (A3) at (9,0) {$C$};
	\node (B3) at (8.5,-1.3) {$G$};
	\node (C3) at (9.5,-1.3) {$E\flat$};
	\draw[->,>=latex] (A3) to node[shift={(-0.3,0.15)}]{$I_7$} (B3) ;
	\draw[->,>=latex] (A3) to node[shift={(0.3,0.15)}]{$I_{3}$} (C3) ;
	
	\draw[->,>=latex, red] (1,-0.6)  to node[above]{${}^{U'W}T_{-2}$} (2,-0.6) ;
	\draw[->,>=latex, red] (4,-0.6)  to node[above]{${}^{WU'}T_{1}$} (5,-0.6) ;
	\draw[->,>=latex, red] (7,-0.6)  to node[above]{${}^{U'U'}T_{1}$} (8,-0.6) ;

	\draw[->,>=latex, violet] (A0) to[bend left=20] node[above]{$T_{-2}$} (A1) ;
	\draw[->,>=latex, violet] (B0) to[bend right=20] node[below]{$T_{3}$} (B1) ;
	\draw[->,>=latex, violet] (C0) to[bend right=20] node[below]{$T_{2}$} (C1) ;

	\draw[->,>=latex, violet] (A1) to[bend left=20] node[above]{$T_{1}$} (A2) ;
	\draw[->,>=latex, violet] (B1) to[bend right=20] node[below]{$T_{-2}$} (B2) ;
	\draw[->,>=latex, violet] (C1) to[bend right=20] node[below]{$T_{-1}$} (C2) ;
	
	\draw[->,>=latex, violet] (A2) to[bend left=20] node[above]{$T_{1}$} (A3) ;
	\draw[->,>=latex, violet] (B2) to[bend right=20] node[below]{$T_{-1}$} (B3) ;
	\draw[->,>=latex, violet] (C2) to[bend right=20] node[below]{$T_{-1}$} (C3) ;
\end{tikzpicture}
\label{subfig:KNet_Berg_2_2}
}
\end{center}
\caption{\subref{subfig:KNet_Berg_2_1} A PK-net interpretation of the last four chords of the progression of Figure \ref{fig:bergop5}. \subref{subfig:KNet_Berg_2_2} Analysis of the chord progression using morphisms of ${(T\text{/}I)}^{\Gamma}$.}
\label{fig:KNet_Berg_2}
\end{figure}
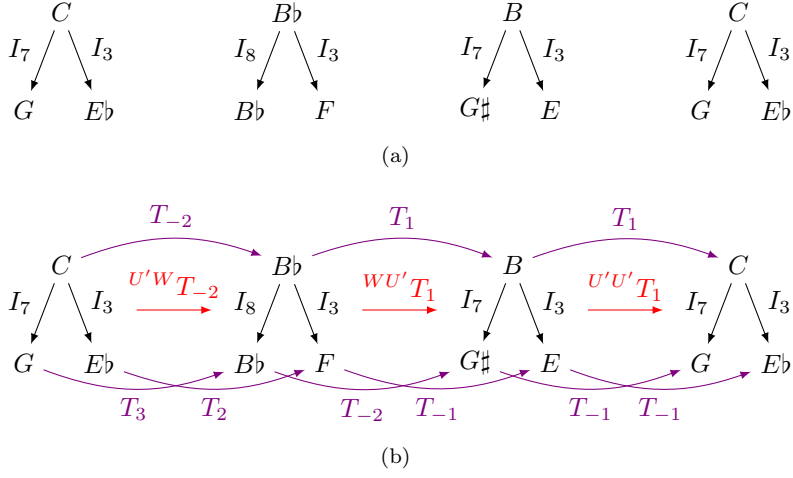

We now consider the hom-set $\text{Hom}(U,V)$ in ${(T\text{/}I)}^{\Gamma}$. Let $\eta \colon U \to V$ be a natural transformation. It is uniquely determined by the component $\eta_X$, which is an element of $T$/$I$. We can derive the components $\eta_Y$ and $\eta_Z$ depending on $\eta_X$.

\begin{itemize}
\item{If $\eta_X=T_p$, with $p$ in $\{0 \ldots 11\}$, then we must have $\eta_Y I_3 = I_4 T_p$, and $\eta_Z I_{10} = I_{10} T_p$. This leads to $\eta_Y = T_{1-p}$, and $\eta_Z = T_{-p}$.}
\item{If $\eta_X=I_p$, with $p$ in $\{0 \ldots 11\}$, then we must have $\eta_Y I_3 = I_4 I_p$, and $\eta_Z I_{10} = I_{10} I_p$. This leads to $\eta_Y = I_{7-p}$, and $\eta_Z = I_{8-p}$.}
\end{itemize}

In a similar way, one can derive the structure of the hom-sets $\text{Hom}(U,U)$, $\text{Hom}(V,V)$, $\text{Hom}(U',W)$, etc. Returning to the PK-nets of Figure \ref{fig:KNet_Berg_1}, it can readily be seen that this progression can be analyzed through the successive application of ${}^{UV}T_{-2}$, ${}^{VV}T_{-1}$, ${}^{VU}T_{2}$, and ${}^{UU}T_{1}$, as shown in Figure \ref{subfig:KNet_Berg_1_2}.

Similarly the progression of the chords represented by the PK-nets of Figure \ref{fig:KNet_Berg_2} can be analyzed through the successive application of ${}^{U'W}T_{-2}$, ${}^{WU'}T_{1}$, and ${}^{U'U'}T_{1}$, as shown in Figure \ref{subfig:KNet_Berg_2_2}. One should observe in particular that ${}^{VU}T_{2} \circ {}^{VV}T_{-1} = {}^{VU}T_{1}$, which has the same components as ${}^{WU'}T_{1}$, evidencing the logic at work behind this progression of chords.

\subsection{Construction of sub-groupoids of $\mathbf{G}^{\Delta}$ and their application in music}

In the examples considered in subsections 1.3 and 1.4, for any object $U$ of the groupoid ${(T\text{/}I)}^{\Gamma}$, the hom-set $\text{Hom}(U,U)$ can be bijectively identified with elements of the $T\text{/}I$ group, and thus contains ``generalized'' transpositions and inversions. For transpositionally-related chords however, it may be useful to consider only a sub-category of ${(T\text{/}I)}^{\Gamma}$ wherein the hom-set $\text{Hom}(U,U)$ only contains transposition-like morphisms. The purpose of this subsection is to show how such a sub-category can be constructed by exploiting the extension structure of the ${T\text{/}I}$ group.

We consider the general case where $\mathbf{G}$ is an extension $1 \to \mathbf{Z} \to \mathbf{G} \to \mathbf{H} \to 1$. This is the case for the $T$/$I$ group for example, which is an extension of the form $1 \to \mathbb{Z}_{12} \to T/I \to \mathbb{Z}_2 \to 1$. The elements of $\mathbf{G}$ can then be uniquely written as $g=(z,h)$ with $z \in \mathbf{Z}$, and $h \in \mathbf{H}$.

Given a poset $\Delta$, we define a functor $\Pi \colon \mathbf{G}^{\Delta} \to \mathbf{H}^{\Delta}$ induced by the homomorphism $\pi \colon \mathbf{G} \to \mathbf{H}$ as follows.

\begin{definition}
For a given poset $\Delta$ with a bottom element, the functor $\Pi \colon \mathbf{G}^{\Delta} \to \mathbf{H}^{\Delta}$ induced by the homomorphism $\pi \colon G \to H$, is the functor which
\begin{itemize}
\item{is the identity on objects, and}
\item{sends any morphism ${}^{FF'}g = {}^{FF'}(z,h)$ in $\mathbf{G}^{\Delta}$ to ${}^{FF'}\pi(g)={}^{FF'}h$ in $\mathbf{H}^{\Delta}$.}
\end{itemize}
\end{definition}

By PropositionÊ \ref{proposition:GDeltaStruct}, we deduce immediately that the functor $\Pi$ is full.

We now consider a sub-category $\widetilde{\mathbf{H}^{\Delta}}$ of $\mathbf{H}^{\Delta}$ such that

\begin{itemize}
\item{for any object $U$ of $\widetilde{\mathbf{H}^{\Delta}}$, $\text{End}(U)$ is trivial, and}
\item{the inclusion functor $\iota \colon \widetilde{\mathbf{H}^{\Delta}} \to \mathbf{H}^{\Delta}$ is the identity on objects.}
\end{itemize}

It is obvious to see that for any objects $U$ and $V$ of $\widetilde{\mathbf{H}^{\Delta}}$, the hom-set $\text{Hom}(U,V)$ is reduced to a singleton which can be identified with one element of $H$. The choice of hom-sets $\text{Hom}(U,V)$ is not unique and determines the sub-category $\widetilde{\mathbf{H}^{\Delta}}$.

We now arrive to the definition of the desired category $\widetilde{\mathbf{G}^{\Delta}}$.

\begin{definition}
The category $\widetilde{\mathbf{G}^{\Delta}}$ is defined as the pull-back of the following diagram.
\begin{center}
\begin{tikzpicture}[scale=1]
	\node (A) at (0,0) {$\widetilde{\mathbf{G}^{\Delta}}$};
	\node (B) at (2,0) {$\widetilde{\mathbf{H}^{\Delta}}$};
	\node (C) at (0,-2) {$\mathbf{G}^{\Delta}$};
	\node (D) at (2,-2) {$\mathbf{H}^{\Delta}$};
	
	\draw[->,dashed,>=latex] (A) to (B) ;
	\draw[->,>=latex] (B) to node[right]{$\iota$} (D) ;
	\draw[->,dashed,>=latex] (A) to (C) ;
	\draw[->,>=latex] (C) to node[above]{$\Pi$} (D) ;
\end{tikzpicture}
\end{center}
\end{definition}

The following propositions are immediate from the definition.

\begin{proposition}
For any object $U$ of $\widetilde{\mathbf{G}^{\Delta}}$, the endomorphism group $\text{End}(U)$ is isomorphic to $\mathbf{Z}$.
\end{proposition}

\begin{proposition}
For any objects $U$ and $V$ of $\widetilde{\mathbf{G}^{\Delta}}$, the hom-set $\text{Hom}(U,V)$ is in bijection with a coset of $\mathbf{Z}$ in $\mathbf{G}$.
\end{proposition}

In the specific case where $\mathbf{G}$ is the $T\text{/}I$ group, there exists a projection functor $\Pi \colon {T\text{/}I}^{\Delta} \to {\mathbb{Z}_2}^{\Delta}$ induced by the homomorphism $\pi \colon T\text{/}I \to \mathbb{Z}_2$, and one can select an appropriate subcategory $\widetilde{{\mathbb{Z}_2}^{\Delta}}$. The subcategory $\widetilde{{T\text{/}I}^{\Delta}}$ obtained by the construction described above is then such that 
\begin{itemize}
\item{for any object $U$ of $\widetilde{{T\text{/}I}^{\Delta}}$, the endomorphism group $\text{End}(U)$ is isomorphic to $\mathbb{Z}_{12}$ and its elements correspond to generalized transpositions as exposed in section 1.3, and}
\item{for any objects $U$ and $V$ of $\widetilde{{T\text{/}I}^{\Delta}}$, the elements of hom-set $\text{Hom}(U,V)$ correspond either to generalized inversions or to generalized transpositions (but not both). Their nature depends on the choice of the subcategory $\widetilde{{\mathbb{Z}_2}^{\Delta}}$.}
\end{itemize}

\section{Groupoid bisections and wreath products}

Wreath products have found many applications in transformational music theory \cite{Peck_2009,Peck_2010}, most notably following the initial work of Hook on Uniform Triadic Transformations (UTT) \cite{Hook_2002}. In this section, we show how groupoids are related to wreath products through \textit{groupoid bisections}, thus generalizing the work of Hook. A more general treatment of groupoid bisections, and their relation to groupoid automorphisms can be found in the annex of this paper.

\subsection{Bisections of a groupoid}

Let $\mathbf{C}$ be a small groupoid. By convention, we will index the objects of $\mathbf{C}$ by $i \in \{1,\ldots,n\}$, where $n$ is the number of objects in $\mathbf{C}$. We denote by $Z$ the group of endomorphisms of any object $i$ of $\mathbf{C}$. We first give the definition of a \textit{bisection of a groupoid}. This notion, which has been studied in the theory of Lie groupoids, is a particular case of the notion of a local section of a topological category as introduced by Ehresmann \cite{Ehresmann_1959}, who later studied the category of such local sections \cite{Ehresmann_1966}. The word \textit{bisection} is due to Mackenzie \cite{mackenzie_1987,mackenzie_2005}.

\begin{definition}
A bisection of $\mathbf{C}$ is the data of a permutation $\sigma \in S_n$ and a collection of morphisms $g_{i\sigma(i)}$ of $\mathbf{C}$ for $i \in \{1,\ldots,n\}$. A bisection will be notated as $$\langle(\ldots,g_{i\sigma(i)},\ldots),\sigma\rangle.$$
\end{definition}

Bisections can be composed according to:

\begin{align*}
\langle(\ldots,f_{i\tau(i)},\ldots),\tau\rangle \circ \langle(\ldots,g_{i\sigma(i)},\ldots),\sigma\rangle \\ = \langle(\ldots,f_{\sigma(i)\tau\sigma(i)}g_{i\sigma(i)},\ldots),\tau\sigma\rangle,
\end{align*}
and form a group $\text{Bis}(\mathbf{C})$. The main result of this section is the following theorem, which establishes the structure of $\text{Bis}(\mathbf{C})$.

\begin{theorem}
The group $\text{Bis}(\mathbf{C})$ is isomorphic to the wreath product $Z \wr S_n$.
\end{theorem}
\begin{proof}
We will first show that $\text{Bis}(\mathbf{C})$ is isomorphic to the semidirect product $Z^n \rtimes S_n$, and then construct an isomorphism from $Z^n \rtimes S_n$ to $Z \wr S_n$.

Let $k$ be an object of $\mathbf{G}$, and let $\{h_{ki}, i \in \{1,\ldots,n\}\}$ be the set obtained by choosing a morphism $h_{ki}$ of $\mathbf{G}$ for every object $i$ of $\mathbf{G}$.
This defines a collection of morphisms $\{h_{ij}=h_{kj}h_{ki}^{-1}, i \in \{1,\ldots,n\}, j \in \{1,\ldots,n\}\}$ such that for any objects $p$, $q$, and $r$ of $\mathbf{G}$, we have $h_{qr}h_{pq}=h_{pr}$.

Let $N$ be the subgroup of $\text{Bis}(\mathbf{C})$ composed of the bisections of $\mathbf{G}$ of the form $\langle(\ldots,n_{ii},\ldots),id\rangle$. The subgroup $N$ is obviously isomorphic to $Z^n$.

Let $H$ be the subgroup of $\text{Bis}(\mathbf{C})$ composed of the bisections of $\mathbf{G}$ of the form $\langle(\ldots,h_{i\sigma(i)},\ldots),\sigma\rangle$. By definition of the morphisms $h_{ij}$, $H$ is obviously isomorphic to $S_n$.

The intersection of the two subgroups $N$ and $H$ is the trivial subgroup composed of the bisection $\langle(\ldots,id_{ii},\ldots),id\rangle.$

Any bisection $\langle(\ldots,g_{i\sigma(i)},\ldots),\sigma\rangle$ can uniquely be decomposed as a product of an element of $N$ and an element of $H$, since any morphism $g_{ij}$ of $\mathbf{G}$ can be written as $g_{ij} = h_{ij}n_{ii}$.

Finally, we need to show that $N$ is normal in $\text{Bis}(\mathbf{C})$. Let $n = \langle(\ldots,n_{ii},\ldots),id\rangle$ be an element of $N$ and $g = \langle(\ldots,g_{i\sigma(i)},\ldots),\sigma\rangle$ be an element of $\text{Bis}(\mathbf{C})$. We have $gn = \langle(\ldots,g_{i\sigma(i)}n_{ii},\ldots),\sigma\rangle.$ Since for any morphisms $g_{ij}$ and $n_{ii}$ of $\mathbf{G}$, the morphism $n'_{jj}= g_{ij}n_{ii}g_{ij}^{-1}$ is uniquely defined, it follows that $gn$ is equal to $n'g$ for a suitable $n'$ of $N$, showing that $N$ is normal in $\text{Bis}(\mathbf{C})$.

For all the reasons above, the group $\text{Bis}(\mathbf{C})$ is isomorphic to the semidirect product $Z^n \rtimes S_n$.

Observe that the morphisms $h_{ij}$ induce automorphisms $\phi_{ij}$ of $Z$ given by $n'_{jj}= h_{ij}n_{ii}h_{ij}^{-1}$, with the added property that for any objects $p$, $q$, and $r$, we have $\phi_{qr} \circ \phi_{pq} = \phi_{pr}$. The isomorphism between $\text{Bis}(\mathbf{C})$ and $Z^n \rtimes S_n$ can be formulated as follows. An element
$$\langle(\ldots,g_{i\sigma(i)},\ldots),\sigma\rangle = \langle(\ldots,h_{i\sigma(i)}n_{ii},\ldots),\sigma\rangle$$
of $\text{Bis}(\mathbf{C})$ can be bijectively identified with the element
$$\langle(\ldots,n_i,\ldots),\sigma\rangle$$
of $Z^n \rtimes S_n$. The composition is given by:
$$\langle(\ldots,m_i,\ldots),\tau\rangle\langle(\ldots,n_i,\ldots),\sigma\rangle = \langle(\ldots,\phi_{\sigma(i)i}(m_{\sigma(i)})n_i,\ldots),\tau\sigma\rangle.$$
It can be checked that the property $\phi_{jk} \circ \phi_{ij} = \phi_{ik}$ ensures that the associativity condition is respected.

We now show that $\text{Bis}(\mathbf{C})$ is isomorphic to the wreath product $Z \wr S_n$.
Consider the map $\chi \colon \text{Bis}(\mathbf{C}) \to G \wr S_n$, which sends an element $\langle(\ldots,m_i,\ldots),\tau\rangle$ of $\text{Bis}(\mathbf{C})$ to the element $\langle(\ldots,\phi_{i1}(m_i),\ldots),\tau\rangle$ of $Z \wr S_n$. We claim that the map $\chi$ is an isomorphism. It is obvious that $\chi$ sends the identity of $\text{Bis}(\mathbf{C})$ to the identity of $Z \wr S_n$.
Let $\langle(\ldots,m_i,\ldots),\tau\rangle$ and $\langle(\ldots,n_i,\ldots),\sigma\rangle$ be two elements of $\text{Bis}(\mathbf{C})$. We have
$$\chi(\langle(\ldots,m_i,\ldots),\tau\rangle) = \langle(\ldots,\phi_{i1}(m_i),\ldots),\tau\rangle,$$
and
$$\chi(\langle(\ldots,n_i,\ldots),\sigma\rangle) = \langle(\ldots,\phi_{i1}(n_i),\ldots),\sigma\rangle.$$
The product of these two elements in $Z \wr S_n$ is equal to
$$\langle(\ldots,\phi_{\sigma(i)1}(m_{\sigma(i)})\phi_{i1}(n_i),\ldots),\tau\sigma\rangle.$$

On the other hand, we have
$$\chi(\langle(\ldots,m_i,\ldots),\tau\rangle\langle(\ldots,n_i,\ldots),\sigma\rangle) \\ = \chi(\langle(\ldots,\phi_{\sigma(i)i}(m_{\sigma(i)})n_i,\ldots),\tau\sigma\rangle),$$
which is equal to
$$\langle(\ldots,\phi_{i1}(\phi_{\sigma(i)i}(m_{\sigma(i)})n_i),\ldots),\tau\sigma\rangle,$$
which, given the properties of the isomorphisms $\phi_{ij}$, is equal to
$$\langle(\ldots,\phi_{\sigma(i)1}(m_{\sigma(i)})\phi_{i1}(n_i),\ldots),\tau\sigma\rangle,$$
thus showing that we have an isomorphism, and that $\text{Bis}(\mathbf{C})$ is isomorphic to the wreath product $Z \wr S_n$.
\end{proof}

\subsection{Application to musical transformations}

The following proposition shows how can one pass from a groupoid action on sets to a corresponding group action.

\begin{proposition}
Let $\mathbf{C}$ be a small groupoid, with $Z$ the group of endomorphisms of its objects, and let $S$ be a functor from $\mathbf{C}$ to $\mathbf{Sets}$.
There is a canonical group action of $Z \wr S_n$ on the disjoint union of the image sets $S(i)$.
\end{proposition}
\begin{proof}
Let $\bigsqcup S(i) = \bigcup \{(x,i), x \in S(i), i \in \{1,\ldots,n\}\}$ be the disjoint union of the image sets $S(i)$ and let
$$\langle(\ldots,g_{i\sigma(i)},\ldots),\sigma\rangle$$
be a bisection of $\mathbf{C}$. The group action of $Z \wr S_n$ on $\bigsqcup S(i)$ is directly given by the action defined as 
$$\langle(\ldots,g_{i\sigma(i)},\ldots),\sigma\rangle \cdot (x,i) = (S(g_{i\sigma(i)})(x),\sigma(i)).$$
\end{proof}

As a direct application to musical transformations, consider a subgroupoid $\widetilde{{T\text{/}I}^{\Delta}}$ as constructed in section 1.5. Given a functor $R \colon \Delta \to \mathbf{Sets}$, we know from Proposition 2 that there exists a canonical functor $P_{R,S} \colon {T\text{/}I}^{\Delta} \to \mathbf{Sets}$, which extends to a functor $\widetilde{P_{R,S}} \colon \widetilde{{T\text{/}I}^{\Delta}} \to \mathbf{Sets}$. From Proposition 5, we thus deduce that there exists a group action of $\mathbb{Z}_{12} \wr S_n$ on the disjoint union of the image sets $\widetilde{P_{R,S}}(i)$, or in other words the set of all PK-nets $(R,S,i,\phi)$. In the case $R$ is representable and $n=2$, it is easy to see that one recovers Hook's UTT group acting on two different types of chords.

\bibliography{pknets_groupoids}
\bibliographystyle{alpha}
\addcontentsline{toc}{section}{References}

\section*{Annex}

This annex gives a more general presentation of groupoid bisections and their relation to groupoid automorphisms. Let $\mathbf{C}$ be a groupoid. We denote by $\mathbf{C_0}$ the collection of its objects, and by $\mathbf{C_1}$ the collection of its morphisms, the functions $s \colon \mathbf{C_1} \to \mathbf{C_0}$ and $t \colon \mathbf{C_1} \to \mathbf{C_0}$ being the usual source and target maps. We denote by $Z$ the group of endomorphisms of any object $e$ of $\mathbf{C}$. A more general definition of a bisection can be formulated as follows. 

\begin{definition}
A bisection of $\mathbf{C}$ is a map $b \colon \mathbf{C_0} \to \mathbf{C_1}$, such that
\begin{itemize}
\item{the map $s \circ b$ is the identity map, and}
\item{the map $t \circ b \colon \mathbf{C_0} \to \mathbf{C_0}$ is a bijection on $\mathbf{C_0}$.}
\end{itemize}
\end{definition}

\begin{proposition}
The bisections of $\mathbf{C}$ form a group $\text{Bis}(\mathbf{C})$ for the composition $b' \star b = b'(t \circ b(e))$ for each object $e$ of $\mathbf{C}$.
\end{proposition}

The inverse of $b$ is $b^{-1} \colon e \to b({(t \circ b)}^{-1}(e))^{-1}$ for each object $e$ of $\mathbf{C}$. The bisections $n$ such that $t \circ n$ is an identity form a subgroup $N$ of $\text{Bis}(\mathbf{C})$ which is isomorphic to the product $\prod_{e}Z$ of the groups $Z$ for each object $e$ of $\mathbf{C}$.

Groupoid bisections are closely related to groupoid automorphisms, as shows the next Proposition.

\begin{proposition}
There exists an homomorphism $\xi$ from $\text{Bis}(\mathbf{C})$ to the group $\text{Aut}(\mathbf{C})$ of automorphisms of $\mathbf{C}$, which associates to a bisection $b$ of $\mathbf{C}$ the automorphism of $\mathbf{C}$ defined by $\xi(b)(g) = b(e')gb(e)^{-1}$ for any morphism $g \colon e \to e'$ in $\mathbf{C}$.
\end{proposition}

The image of $\text{Bis}(\mathbf{C})$ by $\xi$ in $\text{Aut}(\mathbf{C})$ is called the subgroup of \textit{internal automorphisms} of $\mathbf{C}$, notated $\text{Aut}_\text{int}(\mathbf{C})$. If $\mathbf{C}$ is a group, we recover the usual notion of internal automorphism of a group.

We now suppose given an object $u$ of $\mathbf{C}$, and a map $h \colon \mathbf{C_0} \to \mathbf{C_1}$ which associates to each object $e$ of $\mathbf{C}$ a morphism $h(e) \colon e \to u$ such that $h(u)=u$. Then, for each pair $(e,e')$ of objects of $\mathbf{C}$, we denote by $h_{e,e'}$ the morphism $h_{e,e'} = h(e')^{-1}h(e)$. These morphisms from a subgroupoid of $\mathbf{C}$ isomorphic to the groupoid $\mathbf{C_0}^2$ of pairs of objects of $\mathbf{C}$. The following proposition can be found in \cite{Ehresmann_1959}.

\begin{proposition}
Given the above data of an object $u$ of $\mathbf{C}$ and a map $h$, the following statements hold true.
\begin{itemize}
\item{There exists an isomorphism from $\mathbf{C}$ to the groupoid $\mathbf{P}$, product of the groupoid of pairs $\mathbf{C_0}^2$ with the group $Z=\text{End}(u)$. The isomorphism maps $g \colon e \to e'$ on $(e,e',h(e')gh(e)^{-1})$.}
\item{To a bijection $\sigma$ of $\mathbf{C_0}$, the map $h$ associates the bisection $h_\sigma$ mapping $e$ to $h_{e,\sigma(e)}$. The bisections of the form $h_\sigma$ form a subgroup $H$ of $\text{Bis}(\mathbf{C})$ isomorphic to the group $\text{Bij}(\mathbf{C_0})$ of bijections of $\mathbf{C_0}$.}
\end{itemize}
\end{proposition}

The structure of the group $\text{Bis}(\mathbf{C})$ is explicited in the next theorem.

\begin{theorem}
The group $\text{Bis}(\mathbf{C})$ is generated by its subgroups $N=\prod_{e}Z$ and $H \simeq \text{Bij}(\mathbf{C_0})$. Moreover, $H$ acts on $N$ and $\text{Bis}(\mathbf{C})$ is isomorphic to the corresponding semidirect product.
\end{theorem}
\begin{proof}
Any bisection $b$ in $\text{Bis}(\mathbf{C})$ can be written as the composite $b = h_{tb} \circ n_b$, where $n_b$ is an element of $N$ defined by $n_b(e)=h_{tb}(e)^{-1}b(e)$ for each object $e$ of $\mathbf{C}$.

The action of $H$ on $N$ is given by the map $(h_\sigma,n) \mapsto h_\sigma \cdot n$, where we have
$$h_\sigma \cdot n(e)=h_\sigma(e)^{-1}n(\sigma(e))h_\sigma(e).$$

The group $\text{Bis}(\mathbf{C})$ is isomorphic to the corresponding semidirect product $N \rtimes H$, the isomorphism associating a bisection $b$ to the pair $(n_b, h_{tb})$.
\end{proof}

It follows immediately that the group $\text{Aut}_\text{int}(\mathbf{C})$ is isomorphic to the semidirect product $(\prod_{e}Z) \rtimes \text{Bij}(\mathbf{C_0})$.

\end{document}